\documentclass[11pt]{article}%
\usepackage[letterpaper, portrait, margin=1in, headheight=0pt]{geometry}
\usepackage{mathtools}
\usepackage{enumerate}
\usepackage{amsmath,amsthm,amssymb,amsfonts}
\usepackage{enumitem}
\usepackage{mathrsfs}
\usepackage{tikz-cd}
\usepackage{multicol}
\usepackage[title]{appendix}
\usetikzlibrary{quotes,angles}

\DeclareMathOperator{\Hom}{Hom}
\DeclareMathOperator{\Bil}{Bil}

\DeclareMathOperator{\Inf}{Inf}
\DeclareMathOperator{\Tra}{Tra}
\DeclareMathOperator{\Res}{Res}
\DeclareMathOperator{\ima}{Im}

\newcommand{\F}{\mathbb{F}}
\newcommand{\E}{\varepsilon}
\newcommand{\U}{\omega}

\newcommand{\vp}{\varphi}
\newcommand{\LL}{L}

\newcommand{\cZ}{\mathcal{Z}}
\newcommand{\cB}{\mathcal{B}}
\newcommand{\cH}{\mathcal{H}}

\newcommand{\vv}{_{\vdash}}
\newcommand{\dd}{_{\dashv}}
\newcommand{\pp}{_{\perp}}
\newcommand{\FR}{F\lozenge R + R\lozenge F}

\newtheorem{thm}{Theorem}[section]
\newtheorem{lem}[thm]{Lemma}

\newtheorem{cor}[thm]{Corollary}

\theoremstyle{definition}

\newtheorem{ex}[subsection]{Example}

\theoremstyle{remark}

\title{Multipliers and Unicentral Triassociative Algebras}
\author{Erik Mainellis}
\date{}

\begin{document}

\maketitle

\begin{abstract}
    We introduce an analogue of the famous Schur multiplier in the context of associative trialgebras, or triassociative algebras. The latter were first studied by Loday and Ronco in 2001, and are characterized by three operations and eleven relations. The paper highlights an extension-theoretic crossroads of multipliers, covers, and unicentral triassociative algebras. Using theory from previous algebraic contexts as a guide, we develop criteria for when the center of the cover maps onto the center of the algebra. Along the way, we obtain the uniqueness of the cover, two different characterizations of the multiplier, and several exact cohomological sequences.
\end{abstract}

\section{Introduction}
Associative trialgebras, or \textit{triassociative algebras}, and their Koszul dual, tridendriform algebras, were introduced by Loday and Ronco in \cite{loday tri}. Their (co)homology was studied in \cite{yau}, and generalizations of them appeared in \cite{pluri}. Triassociative algebras over a complex vector space were classified for dimensions 1 and 2 in \cite{mainellis tri}. In this paper, we introduce and study the triassociative analogue of the Schur multiplier, a structure that first appeared in the study of group representations \cite{schur}, and its relations with cohomology and extensions. Similar work has been conducted in the context of groups \cite{karp}, Lie algebras \cite{batten}, Leibniz algebras \cite{rogers, mainellis batten}, and diassociative algebras \cite{mainellis batten di}, and multipliers of nilpotent diassociative algebras have also been studied \cite{mainellis nilpotent}. Following the methodology of the diassociative case, we establish a characterization of the multiplier $M(L)$ of a triassociative algebra $L$ in terms of a free presentation and prove that covers of triassociative algebras are unique. We then characterize $M(L)$ by the second cohomology group $\cH^2(L,\F)$ with coefficients in the field and prove a four-part equivalence theorem of conditions for when the center of the cover maps onto the center of the algebra. This happens exactly when the algebra is unicentral. Some of the results in the present paper hold by the same logic as their Lie, Leibniz, or diassociative versions, and so we will sometimes refer to those proofs rather than rewriting them.

\section{Definitions}
Let $\F$ be a field. A triassociative algebra $L = (L,\vdash,\dashv,\perp)$ is an $\F$-vector space equipped with three bilinear products $\vdash,\dashv,\perp:L\times L\xrightarrow{} L$ that satisfy
\begin{align*}
    & (x\vdash y)\vdash z = x\vdash(y\vdash z) && (x\dashv y)\dashv z = x\dashv(y\dashv z) \\ 
    & (x\dashv y)\vdash z = x\vdash (y\vdash z) && (x\dashv y)\dashv z = x\dashv (y\vdash z) \\ 
    & (x\vdash y)\dashv z = x\vdash (y\dashv z) && \text{self} \\
    & (x\perp y)\vdash z = x\vdash(y\vdash z) && (x\dashv y)\dashv z = x\dashv(y\perp z) \\
    & (x\vdash y)\perp z = x\vdash(y\perp z) && (x\perp y)\dashv z = x\perp (y\dashv z) \\
    & (x\dashv y)\perp z = x\perp(y\vdash z) && \text{self} \\
    & (x\perp y)\perp z = x\perp(y\perp z) && \text{self}
\end{align*}
for all $x,y,z\in L$. Here, the columns are chosen to reflect a vertical symmetry that reverses the order of operations and swaps $\vdash$ and $\dashv$. Under this symmetry, three of the axioms are self-reflective. We note that $(L,\vdash,\dashv)$ forms a diassociative algebra, and that each of $(L,\vdash)$, $(L,\dashv)$, and $(L,\perp)$ forms an associative algebra. Subalgebras, ideals, and homomorphisms of triassociative algebras can be defined in the usual way, and the center $Z(L)$ of $L$ is the ideal consisting of all $z\in L$ such that $z*x = x* z = 0$ for all $x\in L$ and $*\in \{\vdash,\dashv,\perp\}$. For any subalgebras $S$ and $T$ of $L$, we denote by $S\lozenge T$ the ideal $S\vdash T + S\dashv T + S\perp T$ and set $S' = S\lozenge S$. Given a pair of triassociative algebras $A$ and $B$, an extension of $A$ by $B$ is a short exact sequence $0\xrightarrow{} A\xrightarrow{} L\xrightarrow{\pi} B\xrightarrow{} 0$ of homomorphisms such that $L$ is a triassociative algebra. One may assume that the homomorphism $A\xrightarrow{} L$ in the extension is the identity map, and we make this assumption throughout the paper. A section of an extension is a linear map $\mu:B\xrightarrow{} L$ such that $\pi\circ \mu = \text{id}_B$, and an extension is called central if $A\subseteq Z(L)$. A central extension is called stem if $A\subseteq L'$.

Let $L$ be a triassociative algebra. We say that a pair $(K,M)$ of triassociative algebras is a defining pair for $L$ if $L\cong K/M$ and $M\subseteq Z(K)\cap K'$. Such a pair is a maximal defining pair if the dimension of $K$ is maximal. In this case, $K$ is called a \textit{cover} of $L$ and $M$ is called the \textit{multiplier} of $L$, denoted by $M(L)$. The multiplier is abelian and thus unique via dimension.

\section{Existence of Universal Elements and Unique Covers}\label{batten 1}
In this section, we obtain the uniqueness of covers for triassociative algebras and characterize the multiplier in terms of a free presentation by following the diassociative case of Section 3 in \cite{mainellis batten di}. The initial pair of lemmas give dimension bounds that are notably different from the Lie, Leibniz, and diassociative cases, as there are simply more possible multiplications for which to account. These lemmas ensure that the members of a defining pair for a finite-dimensional triassociative algebra have bounded dimension.

\begin{lem}\label{dias batten 1.1}
For any triassociative algebra $K$, if $\dim(K/Z(K)) = n$, then $\dim(K')\leq 3n^2$.
\end{lem}

\begin{proof}
Let $\{\overline{x_1},\overline{x_2},\dots, \overline{x_n}\}$ be a basis for $K/Z(K)$. Then $\{x_i\vdash x_j, x_i\dashv x_j, x_i\perp x_j ~|~ 1\leq i,j\leq n\}$ is a generating set for $K'$. Thus $\dim(K')\leq 3n^2$.
\end{proof}

\begin{lem}\label{dias batten 1.2}
Let $L$ be a finite-dimensional triassociative algebra with $\dim L = n$ and let $K$ be the first term in a defining pair for $L$. Then $\dim K\leq n(3n+1)$.
\end{lem}

\begin{proof}
We know that $\dim(K/Z(K)) \leq \dim(K/M) = \dim L = n$ since $M\subseteq Z(K)$. Therefore, $\dim M\leq \dim(K')\leq 3n^2$ via Lemma \ref{dias batten 1.1} since $M\subseteq K'$. We thus have $\dim K = \dim L + \dim M \leq n+3n^2 = n(3n+1)$.
\end{proof}

For the sake of comparison, we provide a table of the analogous dimension bounds for the Lie, Leibniz, associative, and diassociative cases. The subsequent example illustrates that the highest possible dimension bounds of Lemmas \ref{dias batten 1.1} and \ref{dias batten 1.2} can always be obtained.

\begin{center}
    \begin{tabular}{c|c|c}
    Algebra Class & Lemma \ref{dias batten 1.1}; $\dim(K')\leq$ & Lemma \ref{dias batten 1.2}; $\dim K\leq$ \\ \hline
    Lie & $\frac{1}{2}n(n-1)$ & $\frac{1}{2}n(n+1)$ \\ Leibniz & $n^2$ & $n(n+1)$ \\ Associative & $n^2$ & $n(n+1)$ \\ Diassociative & $2n^2$ & $n(2n+1)$ \\ Triassociative & $3n^2$ & $n(3n+1)$
\end{tabular}
\end{center}

\begin{ex}\label{diassociative dimension bound obtained}
Let $L$ be the $n$-dimensional abelian triassociative algebra with basis $\{x_i\}_{i=1,\dots,n}$ and let $M$ be the $3n^2$-dimensional abelian triassociative algebra with basis $\{m_{ij}, s_{ij}, t_{ij}\}_{i,j = 1,\dots n}$. Let $K$ denote the vector space $M\oplus L$ with only nonzero multiplications given by $x_i\vdash x_j = m_{ij}$, $x_i\dashv x_j = s_{ij}$, and $x_i\perp x_j = t_{ij}$ for $i,j=1,\dots, n$. Then $K$ is a triassociative algebra of dimension $n + 3n^2$ and $M = Z(K) = K'$. Clearly $K$ is a cover of $L$ and $M$ is the multiplier since we have maximal possible dimension. Noting that $L=K/Z(K)$, we also obtain $\dim K' = 3n^2$.
\end{ex}

For a finite-dimensional triassociative algebra $L$, let $C(L)$ denote the set of all pairs $(J,\lambda)$ such that $\lambda:J\xrightarrow{} L$ is a surjective homomorphism and $\ker \lambda\subseteq J'\cap Z(J)$. An element $(T,\tau)\in C(L)$ is called a \textit{universal element} in $C(L)$ if, for any $(J,\lambda)\in C(L)$, there exists a homomorphism $\beta:T\xrightarrow{} J$ such that the diagram \[\begin{tikzcd}
T\arrow[r, "\tau"]\arrow[d,swap, "\beta"]& L\\
J \arrow[ur, swap, "\lambda"]
\end{tikzcd}\] commutes, i.e. such that $\lambda\circ \beta = \tau$.

Defining pairs for $L$ correspond to elements of $C(L)$ in a natural way. Indeed, any $(K,\lambda)\in C(L)$ gives rise to a defining pair $(K,\ker \lambda)$. Conversely, any defining pair $(K,M)$ yields a surjective homomorphism $\lambda:K\xrightarrow{} L$ such that $\ker \lambda = M\subseteq Z(K)\cap K'$, and thus $(K,\lambda)\in C(L)$.

As in the previous algebraic contexts, an element $(T,\tau)\in C(L)$ is a universal element if and only if $T$ is a cover. Furthermore, the remaining triassociative analogues of results from Section 3 in \cite{mainellis batten di} follow by similar logic to the diassociative case. We thus obtain the triassociative analogue of the main result.

\begin{thm}\label{dias batten 1.12}
Let $L$ be a finite-dimensional triassociative algebra and let $0\xrightarrow{} R\xrightarrow{} F\xrightarrow{} L\xrightarrow{} 0$ be a free presentation of $L$. Let \[B=\frac{R}{\FR} ~~~~~~ C=\frac{F}{\FR} ~~~~~~ D=\frac{F'\cap R}{\FR}\] Then \begin{enumerate}
    \item All covers of $L$ are isomorphic and have the form $C/E$ where $E$ is the complement to $D$ in $B$.
    \item The multiplier $M(L)$ of $L$ is $D\cong B/E$.
    \item The universal elements in $C(L)$ are the elements $(K,\lambda)$ where $K$ is a cover of $L$.
\end{enumerate}
\end{thm}

\section{Multipliers and Cohomology}\label{batten 3}
In this section, we obtain a characterization of the multiplier of a triassociative algebra in terms of its second cohomology group with coefficients in $\F$. To this end, we begin with a discussion of second cohomology. Consider a central extension $0\xrightarrow{} A\xrightarrow{} L\xrightarrow{} B\xrightarrow{} 0$ of triassociative algebras $A$ by $B$. In \cite{yau}, the author gives the usual interpretation of $\cH^2(B,A)$ via central extensions. As a slight aside, we can generalize this interpretation to the case of noncentral extensions. Following the methodology of \cite{mainellis}, (nonabelian) factor systems of triassociative algebras are 9-tuples of the form \[(\vp\vv,\vp\dd,\vp\pp,\vp\vv',\vp\dd',\vp\pp',f\vv,f\dd,f\pp)\] where \begin{align*}
    &\vp_*:B\xrightarrow{} \mathscr{L}(A) \text{ are linear,} & &\vp_*':B\xrightarrow{} \mathscr{L}(A) \text{ are linear,} & &f_*:B\times B\xrightarrow{} A \text{ are bilinear}
\end{align*} for $*\in \{\vdash,\dashv,\perp\}$. It is a matter of direct computation to verify that factor systems are in one-to-one correspondence with extensions of $A$ by $B$, and must satisfy 77 identities, 7 for each of the 11 triassociative relations. Returning to the central case, all $\vp$ maps become trivial, and these 77 axioms reduce to 11 relations on the $f_*$'s that we will call 2-cocycle conditions. In particular, a 2-cocycle of triassociative algebras $A$ by $B$ is a triple $(f\vv,f\dd,f\pp)$ of bilinear forms $B\times B\xrightarrow{} A$ such that \begin{align*}
    &f\vv(x\vdash y,z) = f\vv(x,y\vdash z) && f\dd(x\dashv y,z) = f\dd(x,y\dashv z) \\ &f\vv(x\dashv y,z) = f\vv(x,y\vdash z) && f\dd(x\dashv y,z) = f\dd(x,y\vdash z) \\ &f\dd(x\vdash y,z) = f\vv(x,y\dashv z) \\ &f\vv(x\perp y,z) = f\vv(x,y\vdash z) && f\dd(x\dashv y,z) = f\dd(x,y\perp z) \\ &f\pp(x\vdash y,z) = f\vv(x,y\perp z) && f\dd(x\perp y,z) = f\pp(x,y\dashv z) \\ &f\pp(x\dashv y,z) = f\pp(x,y\vdash z) \\ &f\pp(x\perp y,z) = f\pp(x,y\perp z)
\end{align*} for all $x,y,z\in B$.

Consider a central extension $0\xrightarrow{} A\xrightarrow{} L\xrightarrow{} B\xrightarrow{} 0$ of $A$ by $B$ with section $\mu:B\xrightarrow{} L$ and define three bilinear forms by \begin{align*}
    &f\vv(x,y) = \mu(x)\vdash\mu(y) - \mu(x\vdash y), \\ &f\dd(x,y) = \mu(x)\dashv\mu(y) - \mu(x\dashv y), \\ &f\pp(x,y) = \mu(x)\perp\mu(y) - \mu(x\perp y)
\end{align*} for $x,y\in B$. Then the images of $f\vv$, $f\dd$, and $f\pp$ fall in $A$ by exactness, and it is readily verified that the triple $(f\vv,f\dd,f\pp)$ is a 2-cocycle of triassociative algebras. Let $\cZ^2(B,A)$ denote the set of all 2-cocycles and $\cB^2(B,A)$ denote the set of all 2-coboundaries, i.e. 2-cocycles $(f\vv,f\dd,f\pp)$ such that \begin{align*}
    &f\vv(x,y) = -\E(x\vdash y),\\ &f\dd(x,y) = -\E(x\dashv y),\\ &f\pp(x,y) = -\E(x\perp y)
\end{align*} for some linear transformation $\E:B\xrightarrow{} A$. As in other algebraic settings, elements $(f\vv,f\dd,f\pp)$ and $(g\vv,g\dd,g\pp)$ in $\cZ^2(B,A)$ correspond to equivalent extensions if and only if they ``differ" by a coboundary, i.e. if there is a linear map $\E:B\xrightarrow{} A$ such that \begin{align*}
    &f\vv(x,y) - g\vv(x,y) = -\E(x\vdash y),\\ & f\dd(x,y) - g\dd(x,y) = -\E(x\dashv y),\\ & f\pp(x,y) - g\pp(x,y) = -\E(x\perp y)
\end{align*} for all $x,y\in B$. Therefore, central extensions of $A$ by $B$ are equivalent if and only if they give rise to the same element of $\cH^2(B,A) = \cZ^2(B,A)/\cB^2(B,A)$.

\subsection{Hochschild-Serre Spectral Sequence}
The remainder of this paper relies on the following Hochschild-Serre type spectral sequence of low dimension. Let $Z$ be a central ideal of a triassociative algebra $L$ and consider the natural central extension \[0\xrightarrow{}Z\xrightarrow{} L \xrightarrow{\beta}L/Z\xrightarrow{} 0\] with section $\mu$ of $\beta$. Let $A$ be a central $L$-module.

\begin{thm}
The sequence \[0\xrightarrow{} \Hom(L/Z,A)\xrightarrow{\Inf_1} \Hom(L,A)\xrightarrow{\Res} \Hom(Z,A)\xrightarrow{\Tra} \cH^2(L/Z,A)\xrightarrow{\Inf_2} \cH^2(L,A)\] is exact.
\end{thm}

We first define the maps in the sequence and verify that they make sense. For any homomorphism $\chi:L/Z\xrightarrow{} A$, define $\Inf_1:\Hom(L/Z,A)\xrightarrow{} \Hom(L,A)$ by $\Inf_1(\chi) = \chi\circ \beta$. Next, for $\pi\in \Hom(L,A)$, define $\Res:\Hom(L,A)\xrightarrow{} \Hom(Z,A)$ by $\Res(\pi) = \pi\circ \iota$ where $\iota:Z\xrightarrow{} L$ is the inclusion map. It is readily verified that $\Inf_1$ and $\Res$ are well-defined and linear. To define the transgression map, let
\begin{align*}
    & f\vv:L/Z\times L/Z\xrightarrow{} Z, \\ & f\dd:L/Z\times L/Z\xrightarrow{} Z, \\ & f\pp:L/Z\times L/Z\xrightarrow{} Z
\end{align*} be the bilinear forms defined by
\begin{align*}
    & f\vv(\overline{x},\overline{y}) = \mu(\overline{x})\vdash \mu(\overline{y}) - \mu(\overline{x}\vdash \overline{y}), \\ & f\dd(\overline{x},\overline{y}) = \mu(\overline{x})\dashv \mu(\overline{y}) - \mu(\overline{x}\dashv \overline{y}), \\ & f\pp(\overline{x},\overline{y}) = \mu(\overline{x})\perp \mu(\overline{y}) - \mu(\overline{x}\perp \overline{y})
\end{align*} for $x,y\in L$. For any $\chi\in \Hom(Z,A)$, we have $(\chi\circ f\vv, \chi\circ f\dd, \chi\circ f\pp)\in \cZ^2(L/Z,A)$ since $\chi$ is a homomorphism. Given another section $\nu$ of $\beta$, define a tuple $(g\vv,g\dd,g\pp)$ of bilinear forms by
\begin{align*}
    & g\vv(\overline{x}, \overline{y}) = \nu(\overline{x})\vdash \nu(\overline{y}) - \nu(\overline{x}\vdash \overline{y}) \\ & g\dd(\overline{x}, \overline{y}) = \nu(\overline{x})\dashv \nu(\overline{y}) - \nu(\overline{x}\dashv \overline{y}) \\ & g\pp(\overline{x}, \overline{y}) = \nu(\overline{x})\perp \nu(\overline{y}) - \nu(\overline{x}\perp \overline{y})
\end{align*} for $x,y\in L$. Then $(f\vv,f\dd,f\pp)$ and $(g\vv,g\dd,g\pp)$ are cohomologous in $\cH^2(L/Z,Z)$, which implies that there exists a linear transformation $\E:L/Z\xrightarrow{} Z$ such that
\begin{align*}
    & f\vv(\overline{x}, \overline{y}) - g\vv(\overline{x}, \overline{y}) = -\E(\overline{x}\vdash \overline{y}), \\ & f\dd(\overline{x}, \overline{y}) - g\dd(\overline{x}, \overline{y}) = -\E(\overline{x}\dashv \overline{y}), \\ & f\pp(\overline{x}, \overline{y}) - g\pp(\overline{x}, \overline{y}) = -\E(\overline{x}\perp \overline{y}).
\end{align*} Therefore, $\chi\circ \E:L/Z\xrightarrow{} A$ is a linear map by which $(\chi\circ f\vv, \chi\circ f\dd, \chi\circ f\pp)$ and $(\chi\circ g\vv, \chi\circ g\dd, \chi\circ g\pp)$ are cohomologous in $\cH^2(L/Z,A)$, and so we define \[\Tra(\chi) = \overline{(\chi\circ f\vv, \chi\circ f\dd, \chi\circ f\pp)}.\] It is straightforward to verify that $\Tra$ is linear.

Finally, we define the second inflation map $\Inf_2:\cH^2(L/Z,A)\xrightarrow{} \cH^2(L,A)$ by \[\Inf_2((f\vv,f\dd,f\pp) + \cB^2(L/Z,A)) = (f\vv',f\dd',f\pp') + \cB^2(L,A)\] where $f\vv'(x,y) = f\vv(\beta(x),\beta(y))$, $f\dd'(x,y) = f\dd(\beta(x),\beta(y))$, and $f\pp'(x,y) = f\pp(\beta(x),\beta(y))$ for $(f\vv,f\dd,f\pp)\in \cZ^2(L/Z,A)$ and $x,y\in L$. It is straightforward to verify that $\Inf_2$ is linear. To check that $\Inf_2$ maps cocycles to cocycles, one begins by computing
\begin{align*}
    f\vv'(x\vdash y, z) & = f\vv(\beta(x\vdash y),\beta(z)) \\ &= f\vv(\beta(x)\vdash \beta(y), \beta(z)) \\ &= f\vv(\beta(x), \beta(y)\vdash \beta(z)) \\ &= f\vv(\beta(x), \beta(y\vdash z)) \\ &= f\vv'(x,y\vdash z)
\end{align*} for all $x,y,z\in L$, which holds since $(f\vv,f\dd,f\pp)$ is a 2-cocycle. The other ten axioms of 2-cocycles hold by similar computations, and so $(f\vv',f\dd',f\pp')\in \cZ^2(L,A)$. To check that $\Inf_2$ maps coboundaries to coboundaries, consider an element $(f\vv,f\dd,f\pp)\in \cB^2(L/Z,A)$. Then there is a linear transformation $\E:L/Z\xrightarrow{} A$ such that $f\vv(\overline{x},\overline{y}) = -\E(\overline{x}\vdash \overline{y})$, $f\dd(\overline{x},\overline{y}) = -\E(\overline{x}\dashv \overline{y})$, and $f\pp(\overline{x},\overline{y}) = -\E(\overline{x}\perp \overline{y})$ for all $x,y\in L$. Here, $\beta(x) = x+Z = \overline{x}$ for any $x\in L$. One has \begin{align*}
    f\vv'(x,y) &= f\vv(\beta(x),\beta(y)) \\ &= -\E(\beta(x)\vdash \beta(y)) \\ &= -\E\circ\beta(x\vdash y)
\end{align*} and, similarly, $f\dd'(x,y) = -\E\circ\beta(x\dashv y)$ and $f\pp'(x,y) = -\E\circ\beta(x\perp y)$. Therefore, $(f\vv',f\dd',f\pp')$ is an element of $\cB^2(L,A)$.

\begin{proof}
Given our section $\mu$ of $0\xrightarrow{} Z\xrightarrow{} L\xrightarrow{\beta}L/Z\xrightarrow{} 0$, let $(f\vv,f\dd,f\pp)\in\cZ^2(L/Z,Z)$ be the cocycle defined by \begin{align*}
    &f\vv(\overline{x},\overline{y}) = \mu(\overline{x})\vdash \mu(\overline{y}) - \mu(\overline{x}\vdash \overline{y}), \\ & f\dd(\overline{x},\overline{y}) = \mu(\overline{x})\dashv \mu(\overline{y}) - \mu(\overline{x}\dashv \overline{y}) \\ & f\pp(\overline{x},\overline{y}) = \mu(\overline{x})\perp \mu(\overline{y}) - \mu(\overline{x}\perp \overline{y})
\end{align*} for $x,y\in L$. We first note that $\Inf_1$ is injective by the same logic as in the previous cases \cite{mainellis batten, mainellis batten di} and thus the sequence is exact at $\Hom(L/Z,A)$. Exactness at $\Hom(L,A)$ also follows similarly.

For exactness at $\Hom(Z,A)$, first consider a homomorphism $\chi\in \Hom(L,A)$. Then \begin{align*}
    \chi\circ f\vv(\overline{x},\overline{y}) &= \chi\circ\mu(\overline{x})\vdash \chi\circ\mu(\overline{y}) - \chi\circ\mu(\overline{x}\vdash\overline{y}) \\ &= -\chi\circ\mu(\overline{x}\vdash\overline{y})
\end{align*} and, similarly, $\chi\circ f\dd(\overline{x},\overline{y}) = -\chi\circ\mu(\overline{x}\dashv\overline{y})$ and $\chi\circ f\pp(\overline{x},\overline{y}) = -\chi\circ\mu(\overline{x}\perp\overline{y})$. We therefore have a coboundary $(\chi\circ f\vv,\chi\circ f\dd,\chi\circ f\pp)\in \cB^2(L/Z,A)$ and \[\Tra(\Res(\chi)) = \Tra(\chi\circ\iota) = \overline{(\chi\circ\iota\circ f\vv,\chi\circ\iota\circ f\dd,\chi\circ\iota\circ f\pp)} = 0.\] Thus $\ima(\Res)\subseteq \ker(\Tra)$. Conversely, suppose there exists a homomorphism $\theta:Z\xrightarrow{} A$ such that \[\Tra(\theta) = \overline{(\theta\circ f\vv, \theta\circ f\dd, \theta\circ f\pp)} = 0,\] i.e. such that $(\theta\circ f\vv,\theta\circ f\dd,\theta\circ f\pp)\in \cB^2(L/Z,A)$. Then there exists a linear map $\E:L/Z\xrightarrow{}A$ such that \begin{align*}
    & \theta\circ f\vv(\overline{x},\overline{y}) = -\E(\overline{x}\vdash \overline{y}), \\ & \theta\circ f\dd(\overline{x},\overline{y}) = -\E(\overline{x}\dashv \overline{y}), \\ & \theta\circ f\pp(\overline{x},\overline{y}) = -\E(\overline{x}\perp \overline{y}).
\end{align*} For any $x,y\in L$, we know that $x=\mu(\overline{x}) + z_x$ and $y=\mu(\overline{y}) + z_y$ for some $z_x,z_y\in Z$. Thus \begin{align*}
    & x\vdash y = \mu(\overline{x}\vdash \overline{y}) + z_{x\vdash y} = \mu(\overline{x})\vdash \mu(\overline{y}), \\ & x\dashv y = \mu(\overline{x}\dashv \overline{y}) + z_{x\dashv y} = \mu(\overline{x})\dashv \mu(\overline{y}), \\ & x\perp y = \mu(\overline{x}\perp \overline{y}) + z_{x\perp y} = \mu(\overline{x})\perp \mu(\overline{y})
\end{align*} which implies that \begin{align}\label{dias chi on H}
    \begin{split}& \theta(z_{x\vdash y}) = \theta(\mu(\overline{x})\vdash\mu(\overline{y}) - \mu(\overline{x}\vdash \overline{y})) = \theta\circ f\vv(\overline{x},\overline{y}) = -\E(\overline{x}\vdash\overline{y}), \\ & \theta(z_{x\dashv y}) = \theta(\mu(\overline{x})\dashv\mu(\overline{y}) - \mu(\overline{x}\dashv \overline{y})) = \theta\circ f\dd(\overline{x},\overline{y}) = -\E(\overline{x}\dashv \overline{y}), \\ & \theta(z_{x\perp y}) = \theta(\mu(\overline{x})\perp\mu(\overline{y}) - \mu(\overline{x}\perp \overline{y})) = \theta\circ f\pp(\overline{x},\overline{y}) = -\E(\overline{x}\perp \overline{y}).\end{split}
    \end{align}
Define a linear map $\sigma:L\xrightarrow{} A$ by $\sigma(x) = \theta(z_x) + \E(\overline{x})$. Then $\sigma(x)\vdash \sigma(y)=0$, $\sigma(x)\dashv\sigma(y) = 0$, and $\sigma(x)\perp\sigma(y) = 0$ since $\ima\sigma\subseteq A$. By (\ref{dias chi on H}), \begin{align*}
    & \sigma(x\vdash y) = \theta(z_{x\vdash y}) + \E(x\vdash y) = 0, \\ & \sigma(x\dashv y) = \theta(z_{x\dashv y}) + \E(x\dashv y) = 0, \\ & \sigma(x\perp y) = \theta(z_{x\perp y}) + \E(x\perp y) = 0.
\end{align*} Thus $\sigma$ is a homomorphism. Moreover, $\sigma(z) = \theta(z) + \E(\overline{z}) = \theta(z)$ for all $z\in Z$, which implies that $\Res(\sigma) = \theta$. Hence $\ker(\Tra)\subseteq \ima(\Res)$ and $\ker(\Tra) = \ima(\Res)$.

For exactness at $\cH^2(L/Z,A)$, first consider a map $\chi\in \Hom(Z,A)$. Then \[\Tra(\chi) = (\chi\circ f\vv,\chi\circ f\dd,\chi\circ f\pp) + \cB^2(L/Z,A)\] where $(\chi\circ f\vv,\chi\circ f\dd,\chi\circ f\pp)\in \cZ^2(L/Z,A)$. One computes \[\Inf_2((\chi\circ f\vv,\chi\circ f\dd,\chi\circ f\pp) + \cB^2(L/Z,A)) = ((\chi\circ f\vv)',(\chi\circ f\dd)',(\chi\circ f\pp)') + \cB^2(L,A)\] where \begin{align*}
    & (\chi\circ f\vv)'(x,y) = \chi\circ f\vv(\overline{x},\overline{y}), \\ & (\chi\circ f\dd)'(x,y) = \chi\circ f\dd(\overline{x},\overline{y}), \\ & (\chi\circ f\pp)'(x,y) = \chi\circ f\pp(\overline{x},\overline{y})
\end{align*} for $x,y\in L$. To show that $\ima(\Tra)\subseteq \ker(\Inf_2)$, we need to find a linear transformation $\E:L\xrightarrow{} A$ such that \begin{align*}
    & (\chi\circ f\vv)'(x,y) = -\E(x\vdash y), \\ & (\chi\circ f\dd)'(x,y) = -\E(x\dashv y), \\ & (\chi\circ f\pp)'(x,y) = -\E(x\perp y).
\end{align*} Let $x=\mu(\overline{x}) + z_x$ and $y=\mu(\overline{y}) + z_y$. Again, the equalities \begin{align*}
    & x\vdash y = \mu(\overline{x}\vdash \overline{y}) + z_{x\vdash y} = \mu(\overline{x})\vdash \mu(\overline{y}), \\ & x\dashv y = \mu(\overline{x}\dashv \overline{y}) + z_{x\dashv y} = \mu(\overline{x})\dashv \mu(\overline{y}), \\ & x\perp y = \mu(\overline{x}\perp \overline{y}) + z_{x\perp y} = \mu(\overline{x})\perp \mu(\overline{y})
\end{align*} yield \begin{align*}
    & \chi\circ f\vv(\overline{x},\overline{y}) = \chi(\mu(\overline{x}\vdash\overline{y} - \mu(\overline{x}\vdash\overline{y})) = \chi(z_{x\vdash y}), \\ & \chi\circ f\dd(\overline{x},\overline{y}) = \chi(\mu(\overline{x}\dashv\overline{y} - \mu(\overline{x}\dashv\overline{y})) = \chi(z_{x\dashv y}), \\ & \chi\circ f\pp(\overline{x},\overline{y}) = \chi(\mu(\overline{x}\perp\overline{y} - \mu(\overline{x}\perp\overline{y})) = \chi(z_{x\perp y}).
\end{align*} Define $\E(x) = -\chi(z_x)$. Then $\E$ is linear and \begin{align*}
    & \E(x\vdash y) = -\chi(z_{x\vdash y}) = -\chi\circ f\vv(\overline{x},\overline{y}) = -(\chi\circ f\vv)'(x,y), \\ & \E(x\dashv y) = -\chi(z_{x\dashv y}) = -\chi\circ f\dd(\overline{x},\overline{y}) = -(\chi\circ f\dd)'(x,y), \\ & \E(x\perp y) = -\chi(z_{x\perp y}) = -\chi\circ f\pp(\overline{x},\overline{y}) = -(\chi\circ f\pp)'(x,y).
\end{align*} This implies that $((\chi\circ f\vv)',(\chi\circ f\dd)',(\chi\circ f\pp)') \in \cB^2(L,A)$ and hence $\ima(\Tra)\subseteq \ker(\Inf_2)$.

Conversely, suppose $(g\vv,g\dd,g\pp)\in \cZ^2(L/Z,A)$ is a cocycle such that \[\overline{(g\vv,g\dd,g\pp)}\in \ker(\Inf_2).\] Then there exists a linear transformation $\E:L\xrightarrow{} A$ such that \begin{align*}
    & g\vv(\overline{x},\overline{y}) = g\vv'(x,y) = -\E(x\vdash y), \\ & g\dd(\overline{x},\overline{y}) = g\dd'(x,y) = -\E(x\dashv y), \\ & g\pp(\overline{x},\overline{y}) = g\pp'(x,y) = -\E(x\perp y)
\end{align*} for all $x,y\in L$. Since $\E$ is linear, $(\E\circ f\vv,\E\circ f\dd,\E\circ f\pp)\in \cZ^2(L/Z,A)$. As before, $x=\mu(\overline{x}) + z_x$ and $y=\mu(\overline{y}) + z_y$ for some $z_x,z_y\in Z$. Therefore \begin{align*}
    & x\vdash y = \mu(\overline{x})\vdash\mu(\overline{y}), \\ & x\dashv y = \mu(\overline{x})\dashv \mu(\overline{y}), \\ & x\perp y = \mu(\overline{x})\perp \mu(\overline{y}).
\end{align*} Now \begin{align*}
    g\vv'(x,y) &= g\vv(\overline{x},\overline{y}) \\ &= -\E(x\vdash y)\\ &= -\E(\overline{x}\vdash\overline{y}) \\ &= - \E\circ f\vv(\overline{x},\overline{y}) -\E\circ \mu(\overline{x}\vdash\overline{y})
\end{align*} where $\E\circ \mu:L/Z\xrightarrow{} A$. By similar computations, we have $g\dd'(x,y) = -\E\circ f\dd(\overline{x},\overline{y}) - \E\circ \mu(\overline{x}\dashv \overline{y})$ and $g\pp'(x,y) = -\E\circ f\pp(\overline{x},\overline{y}) - \E\circ \mu(\overline{x}\perp \overline{y})$. Therefore \[\overline{(g\vv,g\dd,g\pp)} = \overline{(-\E\circ f\vv,-\E\circ f\dd,-\E\circ f\pp)} = -\Tra(\E)\] which implies that $\ker(\Inf_2)\subseteq \ima(\Tra)$.
\end{proof}

\subsection{Relation of Multipliers and Cohomology}
Let $L$ be a triassociative algebra and let $\F$ be considered as a central $L$-module. The following theorem holds similarly to its diassociative and Leibniz analogues.

\begin{thm}\label{dias if tra surj}
Let $Z$ be a central ideal in $L$. Then $L'\cap Z$ is isomorphic to the image of $\Hom(Z,\F)$ under the transgression map. In particular, if $\Tra$ is surjective, then $L'\cap Z\cong \cH^2(L/Z,\F)$.
\end{thm}

Now consider a free presentation $0\xrightarrow{} R\xrightarrow{} F\xrightarrow{\U} L\xrightarrow{} 0$ of $L$. The sequence \[0\xrightarrow{} \frac{R}{\FR}\xrightarrow{} \frac{F}{\FR}\xrightarrow{} L\xrightarrow{} 0\] is a central extension since all of $R*F$ and $F*R$ are contained in $\FR$ for $*\in \{\vdash,\dashv,\perp\}$.

\begin{lem}\label{dias restriction to R}
Let $0\xrightarrow{} A\xrightarrow{} B\xrightarrow{\phi} C\xrightarrow{} 0$ be a central extension and $\alpha:L\xrightarrow{} C$ be a homomorphism. Then there exists a homomorphism $\beta:F/(\FR)\xrightarrow{} B$ such that \[\begin{tikzcd}
0\arrow[r]& \frac{R}{\FR}\arrow[r] \arrow[d, "\gamma"] & \frac{F}{\FR} \arrow[r] \arrow[d,"\beta"] & \LL\arrow[r] \arrow[d, "\alpha"] &0 \\
0\arrow[r] &A \arrow[r] & B\arrow[r] &C\arrow[r] &0
\end{tikzcd}\] is commutative, where $\gamma$ is the restriction of $\beta$ to $R/(\FR)$.
\end{lem}

\begin{proof}
Since $F$ is free, there exists a homomorphism $\sigma:F\xrightarrow{} B$ such that \[\begin{tikzcd}
 F\arrow[r, "\U"] \arrow[d, "\sigma", swap] & \LL \arrow[d,"\alpha"] \\ B \arrow[r, "\phi"] & C \end{tikzcd}\] is commutative. Let $r\in R\subseteq F$. Then $\U(r) = 0$ since $\ker \U = R$. Therefore $0=\alpha\circ \U(r) = \phi\circ\sigma(r)$ and so $\sigma(R)\subseteq \ker \phi$. We want to show that $\FR\subseteq \ker \sigma$. If $x\in F$ and $r\in R$, then \begin{align*}
     &\sigma(r* x) = \sigma(r)*\sigma(x) = 0, && \sigma(x*r) = \sigma(x)*\sigma(r) = 0
 \end{align*} for $*\in\{\vdash,\dashv,\perp\}$ since $\sigma(r)\in \ker \phi = A \subseteq Z(B)$. Thus $\sigma$ induces the desired homomorphism $\beta$.
\end{proof}

\begin{lem}\label{dias tra surj}
Let $0\xrightarrow{} R\xrightarrow{} F\xrightarrow{} L\xrightarrow{} 0$ be a free presentation of $L$ and let $A$ be a central $L$-module. Then the transgression map $\Tra:\Hom(R/(\FR),A)\xrightarrow{} \cH^2(L,A)$ associated with \[0\xrightarrow{} \frac{R}{\FR}\xrightarrow{} \frac{F}{\FR}\xrightarrow{\phi} L\xrightarrow{} 0\] is surjective.
\end{lem}

\begin{proof}
Let $\overline{(g\vv,g\dd,g\pp)}\in \cH^2(L,A)$ with associated central extension $0\xrightarrow{} A\xrightarrow{} E\xrightarrow{\vp} L\xrightarrow{} 0$. By the previous lemma, there exists a homomorphism $\theta$ such that \[\begin{tikzcd}
0\arrow[r]& \frac{R}{\FR}\arrow[r] \arrow[d, "\gamma"] & \frac{F}{\FR} \arrow[r, "\phi"] \arrow[d,"\theta"] & \LL\arrow[r] \arrow[d, "\text{id}"] &0 \\
0\arrow[r] &A \arrow[r] & E\arrow[r, "\vp"] &\LL\arrow[r] &0
\end{tikzcd}\] is commutative and $\gamma = \theta|_{R/(\FR)}$. Let $\mu$ be a section of $\phi$. Then $\vp\circ\theta\circ \mu = \phi\circ \mu = \text{id}$ and so $\theta\circ \mu$ is a section of $\vp$. Let $\lambda = \theta\circ \mu$ and define \begin{align*} &\beta\vv(x,y) = \lambda(x)\vdash\lambda(y) - \lambda(x\vdash y), \\ & \beta\dd(x,y) = \lambda(x)\dashv\lambda(y) - \lambda(x\dashv y), \\ & \beta\pp(x,y) = \lambda(x)\perp\lambda(y) - \lambda(x\perp y).
\end{align*} Then $(\beta\vv,\beta\dd,\beta\pp)\in \cZ^2(\LL,A)$ and $(\beta\vv,\beta\dd,\beta\pp)$ is cohomologous with $(g\vv,g\dd,g\pp)$ since they are associated with the same extension. One computes \begin{align*}
    \beta\vv(x,y) &= \theta(\mu(x))\vdash\theta(\mu(y)) - \theta(\mu(x\vdash y)) \\ &= \theta(\mu(x)\vdash\mu(y) - \mu(x\vdash y))\\ &= \gamma(\mu(x)\vdash\mu(y) - \mu(x\vdash y)) \\ &= \gamma(f\vv(x,y))
\end{align*} where $f\vv(x,y) = \mu(x)\vdash\mu(y) - \mu(x\vdash y)$ and since $\gamma = \theta|_{R/(\FR)}$. Similarly, one computes \begin{align*}
    & \beta\dd(x,y) = \gamma(f\dd(x,y)), \\ & \beta\pp(x,y) = \gamma(f\pp(x,y))
\end{align*} for $f\dd(x,y) = \mu(x)\dashv\mu(y) - \mu(x\dashv y)$ and $f\pp(x,y) = \mu(x)\perp\mu(y) - \mu(x\perp y)$. Thus \[\Tra(\gamma) = \overline{(\gamma\circ f\vv,\gamma\circ f\dd,\gamma\circ f\pp)} = \overline{(\beta\vv,\beta\dd,\beta\pp)} = \overline{(g\vv,g\dd,g\pp)}\] and $\Tra$ is surjective.
\end{proof}

\begin{lem}\label{dias set lemma}
If $C\subseteq A$ and $C\subseteq B$, then $A/C\cap B/C = (A\cap B)/C$.
\end{lem}

\begin{proof}
Follows by the same logic as the Lie analogue \cite{batten}.
\end{proof}

\begin{thm}
Let $L$ be a triassociative algebra over a field $\F$ and $0\xrightarrow{} R\xrightarrow{} F\xrightarrow{} L\xrightarrow{} 0$ be a free presentation of $L$. Then \[\cH^2(L,\F) \cong \frac{F'\cap R}{\FR}.\] In particular, if $L$ is finite-dimensional, then $M(L)\cong \cH^2(L,\F)$.
\end{thm}

\begin{proof}
Let $\overline{R} = \frac{R}{\FR}$ and $\overline{F} = \frac{F}{\FR}$. Then $0\xrightarrow{} \overline{R}\xrightarrow{} \overline{F}\xrightarrow{} \LL\xrightarrow{} 0$ is a central extension. By Lemma \ref{dias tra surj}, $\Tra:\Hom(\overline{R},\F)\xrightarrow{} \cH^2(\LL,\F)$ is surjective. By Theorem \ref{dias if tra surj}, \[\overline{F}'\cap \overline{R} \cong \cH^2(\overline{F}/\overline{R},\F) \cong \cH^2(\LL,\F).\] By Lemma \ref{dias set lemma}, \[\overline{F}'\cap \overline{R} \cong \frac{F'}{\FR} \cap \frac{R}{\FR} = \frac{F'\cap R}{\FR}.\] Therefore, \[M(\LL) = \frac{F'\cap R}{\FR} \cong \cH^2(\LL,\F)\] by the characterization of $M(L)$ from Theorem \ref{dias batten 1.12}.
\end{proof}

\section{Unicentral Algebras}\label{batten 4}
For a triassociative algebra $L$, let $Z^*(L)$ denote the intersection of all images $\U(Z(E))$ such that $0\xrightarrow{} \ker \U\xrightarrow{} E\xrightarrow{\U} L\xrightarrow{} 0$ is a central extension of $L$. It is easy to see that $Z^*(L)\subseteq Z(L)$. We say that a triassociative algebra $L$ is \textit{unicentral} if $Z(L) = Z^*(L)$. In this section, we develop criteria for when the center of the cover of $L$ maps onto the center of $L$. One of these criteria will take the form of when $Z(L)\subseteq Z^*(L)$, or when the algebra is unicentral.

\subsection{More Sequences}
We first extend our Hochschild-Serre sequence. Given a central ideal $Z$ in $L$, consider the central extension $0\xrightarrow{} Z\xrightarrow{} L\xrightarrow{} L/Z\xrightarrow{} 0$. To define our extension map, let $(f\vv',f\dd',f\pp')\in \cZ^2(L,\F)$ and define six bilinear forms \begin{align*}
    & f\vv'':L/L'\times Z\xrightarrow{} \F, && f\dd'':L/L'\times Z\xrightarrow{} \F, && f\pp'':L/L'\times Z\xrightarrow{} \F, \\ & g\vv'':Z\times L/L'\xrightarrow{} \F, && g\dd'':Z\times L/L'\xrightarrow{} \F, && g\pp'':Z\times L/L'\xrightarrow{} \F
\end{align*} by \begin{align*}
    & f\vv''(x+L',z) = f\vv'(x,z), && f\dd''(x+L',z) = f\dd'(x,z), && f\pp''(x+L',z) = f\pp'(x,z), \\
    & g\vv''(z,x+L') = f\vv'(z,x), && g\dd''(z,x+L') = f\dd'(z,x), && g\pp''(z,x+L') = f\pp'(z,x)
\end{align*} for $x\in L$, $z\in Z$. To check that these four maps are well-defined, one computes \begin{align*}
    & f\vv''(x* y+L',z) = 0, &&f\dd''(x* y+L',z) = 0, && f\pp''(x* y+L',z) = 0, \\ & g\vv''(z,x* y+L') = 0, && g\dd''(z,x* y+L') = 0, && g\pp''(z,x* y+L') = 0
\end{align*} for $*\in \{\vdash,\dashv,\perp\}$ via the identities of $(f\vv',f\dd',f\pp')$ and since $z\in Z(L)$. In short, the triassociative cocycle conditions ensure that there is always a way to associate $z$ with one of $x$ or $y$, making each bilinear form return zero. Hence \begin{align*}
    (f\vv'',g\vv'',f\dd'',g\dd'',f\pp'',g\pp'') &\in (\Bil(L/L'\times Z, \F)\oplus \Bil(Z\times L/L', \F))^3 \\ &\cong (L/L'\otimes Z \oplus Z\otimes L/L')^3.
\end{align*} Now consider $(f\vv',f\dd',f\pp')\in \cB^2(L,\F)$. Then there exists a linear transformation $\E:L\xrightarrow{} \F$ such that $f\vv'(x,y) = -\E(x\vdash y)$, $f\dd'(x,y) = -\E(x\dashv y)$, and $f\pp'(x,y) = -\E(x\perp y)$ for all $x,y\in L$. One computes $f_*''(x+L',z) = f_*'(x,z) = -\E(x* z) = 0$ and $g_*''(z,x+L') = f_*'(z,x) = -\E(z* x) = 0$ for $*\in \{\vdash,\dashv,\perp\}$ since $z\in Z(L)$. Hence, a map \[\delta:(f\vv',f\dd',f\pp')+\cB^2(L,\F)\mapsto (f\vv'',g\vv'',f\dd'',g\dd'',f\pp'',g\pp'')\] is induced that is clearly linear since $f\vv'$, $f\dd'$, $f\pp'$, $f\vv''$, $g\vv''$, $f\dd''$, $g\dd''$, $f\pp''$, $g\pp''$ are all in vector spaces of bilinear forms and the latter six are defined by the first three.

\begin{thm}\label{dias batten 4.1}
Let $Z$ be a central ideal of a triassociative algebra $L$. The sequence \[\cH^2(L/Z,\F)\xrightarrow{\Inf} \cH^2(L,\F)\xrightarrow{\delta} (L/L'\otimes Z \oplus Z\otimes L/L')^3\] is exact.
\end{thm}

\begin{proof}
Consider $(f\vv,f\dd,f\pp)\in \cZ^2(L/Z,\F)$. Then \[\Inf((f\vv,f\dd,f\pp) + \cB^2(L/Z, \F)) = (f\vv',f\dd',f\pp') + \cB^2(L,\F)\] where $f\vv'(x,y) = f\vv(x+Z,y+Z)$, $f\dd'(x,y) = f\dd(x+Z,y+Z)$, and $f\pp'(x,y) = f\pp(x+Z,y+Z)$ for $x,y\in L$. Moreover, \[\delta((f\vv',f\dd',f\pp')+\cB^2(L,\F)) = (f\vv'',g\vv'',f\dd'',g\dd'',f\pp'',g\pp'')\] where
\begin{align*}
    & f\vv''(x+L',z) = f\vv'(x,z) = f\vv(x+Z,z+Z) = 0, \\
    & g\vv''(z,x+L') = f\vv'(z,x) = f\vv(z+Z,x+Z) = 0, \\ &f\dd''(x+L',z) = f\dd'(x,z) = f\dd(x+Z,z+Z) = 0, \\
    & g\dd''(z,x+L') = f\dd'(z,x) = f\dd(z+Z,x+Z) = 0, \\ & f\pp''(x+L',z) = f\pp'(x,z) = f\pp(x+Z,z+Z) = 0, \\
    & g\pp''(z,x+L') = f\pp'(z,x) = f\pp(z+Z,x+Z) = 0 \\
\end{align*} for all $x\in L$ and $z\in Z$. Thus \begin{align*}
    \delta(\Inf((f\vv,f\dd,f\pp) + \cB^2(L/Z,\F))) &= \delta((f\vv',f\dd',f\pp')+\cB^2(L,\F)) \\ &= (f\vv'',g\vv'',f\dd'',g\dd'',f\pp'',g\pp'') \\ &= (0,0,0,0,0,0)
\end{align*} which implies that $\ima(\Inf)\subseteq \ker \delta$.

Conversely, suppose $\delta((f\vv',f\dd',f\pp')+\cB^2(L,\F)) = (f\vv'',g\vv'',f\dd'',g\dd'',f\pp'',g\pp'') = (0,0,0,0,0,0)$ for some cocycle $(f\vv',f\dd',f\pp')\in \cZ^2(L,\F)$. In other words, $0 = f_*''(x+L',z) = f_*'(x,z)$ and $0 = g_*''(z,x+L') = f_*'(z,x)$ for all $x\in L$, $z\in Z$, and $*\in \{\vdash,\dashv,\perp\}$. Hence \begin{align*}
    & f\vv'(x+z,y+z') = f\vv'(x,y) + f\vv'(x,z') + f\vv'(z,y) + f\vv'(z,z') = f\vv'(x,y), \\ & f\dd'(x+z,y+z') = f\dd'(x,y) + f\dd'(x,z') + f\dd'(z,y) + f\dd'(z,z') = f\dd'(x,y), \\ & f\pp'(x+z,y+z') = f\pp'(x,y) + f\pp'(x,z') + f\pp'(z,y) + f\pp'(z,z') = f\pp'(x,y)
\end{align*} for all $z,z'\in Z$, which implies that the bilinear forms \begin{align*}
    g\vv:L/Z\times L/Z\xrightarrow{} \F, && g\dd:L/Z\times L/Z\xrightarrow{} \F, && g\pp:L/Z\times L/Z\xrightarrow{} \F
\end{align*} defined by \begin{align*}
    g\vv(x+Z,y+Z) = f\vv'(x,y), && g\dd(x+Z,y+Z) = f\dd'(x,y), && g\pp(x+Z,y+Z) = f\pp'(x,y)
\end{align*} are well-defined. Furthermore, $(g\vv,g\dd,g\pp)\in \cZ^2(L/Z,\F)$ since $(f\vv',f\dd',f\pp')$ is a cocycle. Thus \[\Inf((g\vv,g\dd,g\pp) + \cB^2(L/Z,\F)) = (f\vv',f\dd',f\pp')+\cB^2(L,\F)\] which implies that $\ker \delta \subseteq \ima(\Inf)$.
\end{proof}

One of our criteria will involve this $\delta$ map. Another will involve the \textit{natural map} $\beta$ that appears in the following analogue of the Ganea sequence.

\begin{thm}\label{dias batten 4.2}
Let $Z$ be a central ideal of a finite-dimensional triassociative algebra $L$. Then the sequence \[(L/L'\otimes Z \oplus Z\otimes L/L')^3 \xrightarrow{} M(L) \xrightarrow{\beta} M(L/Z)\xrightarrow{} L'\cap Z\xrightarrow{} 0\] is exact.
\end{thm}

\begin{proof}
Let $F$ be a free triassociative algebra such that $L=F/R$ and $Z=T/R$ for ideals $T$ and $R$ of $F$. Since $Z\subseteq Z(L)$, we have $T/R\subseteq Z(F/R)$, and thus $F\lozenge T + T\lozenge F\subseteq R$. The exactness of \[M(L)\xrightarrow{\beta} M(L/Z)\xrightarrow{\gamma} L'\cap Z\xrightarrow{} 0\] follows similarly to the diassociative case. It remains to show that \[(L/L'\otimes Z \oplus Z\otimes L/L')^3 \xrightarrow{} M(L) \xrightarrow{\beta} M(L/Z)\] is exact. Define six maps \begin{align*}
    & \theta\vv:\frac{T}{R}\times \frac{F}{R+F'}\xrightarrow{} \frac{R\cap F'}{\FR}, && \alpha\vv:\frac{F}{R+F'}\times \frac{T}{R}\xrightarrow{} \frac{R\cap F'}{\FR}, \\ & \theta\dd:\frac{T}{R}\times \frac{F}{R+F'}\xrightarrow{} \frac{R\cap F'}{\FR}, && \alpha\dd:\frac{F}{R+F'}\times, \frac{T}{R}\xrightarrow{} \frac{R\cap F'}{\FR}, \\ & \theta\pp:\frac{T}{R}\times \frac{F}{R+F'}\xrightarrow{} \frac{R\cap F'}{\FR}, && \alpha\pp:\frac{F}{R+F'}\times \frac{T}{R}\xrightarrow{} \frac{R\cap F'}{\FR}
\end{align*} by \begin{align*}
    & \theta\vv(\overline{t},\overline{f}) = t\vdash f + (\FR), && \alpha\vv(\overline{f},\overline{t}) = f\vdash t+ (\FR), \\ & \theta\dd(\overline{t},\overline{f}) = t\dashv f+ (\FR), && \alpha\dd(\overline{f},\overline{t}) = f\dashv t + (\FR), \\ & \theta\pp(\overline{t},\overline{f}) = t\perp f + (\FR), && \alpha\pp(\overline{f},\overline{t}) = f\perp t+ (\FR)
\end{align*} for $t\in T$, $f\in F$. These maps are bilinear since multiplication is bilinear. To check that they are well-defined, suppose $(t+R, f+(R+F')) = (t'+R, f'+(R+F'))$ for $t,t'\in T$ and $f,f'\in F$. Then $t-t'\in R$ and $f-f'\in R+F'$, which implies that $t=t'+r$ and $f=f'+x$ for some $r\in R$ and $x\in R+F'$. One computes \begin{align*}
    t\dashv f - t'\dashv f' &= (t'+r)\dashv (f'+x) - t'\dashv f' \\ &= r\dashv f' + r\dashv x + t'\dashv x \\ &\in (R\dashv F) + (R\dashv F) + (T\dashv R + T\dashv F')
\end{align*} which is contained in $\FR$ since \begin{align*}
    T\dashv F' &= T\dashv (F\vdash F) + T\dashv (F\dashv F) + T\dashv (F\perp F) \\ &= (T\dashv F)\dashv F + (T\dashv F)\dashv F + (T\dashv F)\dashv F
\end{align*} and $T\dashv F\subseteq R$. Next, \begin{align*} t\vdash f - t'\vdash f' &= (t'+r)\vdash (f'+x) - t'\vdash f' \\ &= r\vdash f' + r\vdash x + t'\vdash x \\ &\in (R\vdash F) + (R\vdash F) + (T\vdash R + T\vdash F')
\end{align*} which is also contained in $\FR$ since \begin{align*}
    T\vdash F' &= T\vdash (F\vdash F) + T\vdash(F\dashv F) + T\vdash(F\perp F) \\ &= (T\vdash F)\vdash F +(T\vdash F)\dashv F + (T\vdash F)\perp F
\end{align*} and $T\vdash F\subseteq R$. Finally, \begin{align*}
    t\perp f - t'\perp f' & = (t'+r)\perp(f'+x) - t'\perp f' \\ &= r\perp f' + r\perp x + t'\perp x \\ & \in (R\perp F) + (R\perp F) + (T\perp R + T\perp F')
\end{align*} which is contained in $\FR$ since \begin{align*}
    T\perp F' &= T\perp (F\vdash F) + T\perp(F\dashv F) + T\perp(F\perp F) \\ &= (T\dashv F)\perp F + (T\perp F)\dashv F + (T\perp F)\perp F
\end{align*} and $T\perp F$ is also contained in $R$. Expressions $f*t - f'*t'$ for $*\in\{\vdash,\dashv,\perp\}$ fall in $\FR$ by similar manipulations. Thus, our bilinear forms $\theta_*$ and $\alpha_*$ are well-defined, and so induce linear maps \begin{align*}
    & \overline{\theta\vv}:\frac{T}{R}\otimes \frac{F}{R+F'}\xrightarrow{} \frac{R\cap F'}{\FR}, && \overline{\alpha\vv}:\frac{F}{R+F'}\otimes \frac{T}{R}\xrightarrow{} \frac{R\cap F'}{\FR}, \\ & \overline{\theta\dd}:\frac{T}{R}\otimes \frac{F}{R+F'}\xrightarrow{} \frac{R\cap F'}{\FR}, && \overline{\alpha\dd}:\frac{F}{R+F'}\otimes \frac{T}{R}\xrightarrow{} \frac{R\cap F'}{\FR}, \\ & \overline{\theta\pp}:\frac{T}{R}\otimes \frac{F}{R+F'}\xrightarrow{} \frac{R\cap F'}{\FR}, && \overline{\alpha\pp}:\frac{F}{R+F'}\otimes \frac{T}{R}\xrightarrow{} \frac{R\cap F'}{\FR}.
\end{align*} These, in turn, yield a linear transformation \[\overline{\theta}:\left(\frac{F}{R+F'}\otimes \frac{T}{R} \oplus \frac{T}{R}\otimes \frac{F}{R+F'}\right)^3 \xrightarrow{} \frac{R\cap F'}{\FR}\] defined by $\overline{\theta}(a,b,c,d,e,f) = \overline{\alpha\vv}(a) + \overline{\theta\vv}(b) + \overline{\alpha\dd}(c) + \overline{\theta\dd}(d) + \overline{\alpha\pp}(e) + \overline{\theta\pp}(f)$. The image of $\overline{\theta}$ is \[\frac{F\lozenge T+T\lozenge F}{\FR}\] which is precisely equal to $\{x + (\FR)~|~ x\in R\cap F',~ x\in F\lozenge T + T\lozenge F\} = \ker \beta$.
\end{proof}

\begin{cor}
(Stallings Sequence) Let $Z$ be a central ideal in a finite-dimensional triassociative algebra $L$. Then the sequence \[M(L)\xrightarrow{} M(L/Z)\xrightarrow{} Z\xrightarrow{} L/L'\xrightarrow{} \frac{L}{Z+L'}\xrightarrow{} 0\] is exact.
\end{cor}

\begin{proof}
Follows by the same logic as the diassociative case.
\end{proof}

\subsection{Main Results}
The remainder of the paper follows similarly to its diassociative analogue \cite{mainellis batten di}. First, consider a free presentation $0\xrightarrow{} R\xrightarrow{} F\xrightarrow{\pi} L\xrightarrow{} 0$ of a finite-dimensional triassociative algebra $L$ and let $\overline{X}$ denote the quotient algebra $\frac{X}{\FR}$ for any $X$ such that $\FR\subseteq X\subseteq F$. Since $R=\ker \pi$ and $\FR\subseteq R$, $\pi$ induces a homomorphism $\overline{\pi}:\overline{F}\xrightarrow{} L$ such that the diagram \[\begin{tikzcd}
F\arrow[r,"\pi"]\arrow[d]& L\\
\overline{F} \arrow[ur, swap, "\overline{\pi}"]
\end{tikzcd}\] commutes. Since $\overline{R}\subseteq Z(\overline{F})$, there exists a complement $\frac{S}{\FR}$ to $\frac{R\cap F'}{\FR}$ in $\frac{R}{\FR}$ where $S\subseteq R\subseteq \ker \pi$ and $\overline{S}\subseteq \overline{R}\subseteq \ker\overline{\pi}$. Thus $\overline{\pi}$ induces a homomorphism $\pi_S:F/S\xrightarrow{} L$ such that the extension \[0\xrightarrow{} R/S\xrightarrow{} F/S\xrightarrow{\pi_S} L\xrightarrow{} 0\] is central. This extension is stem since \[R/S\cong \frac{R\cap F'}{\FR} = \ker \pi_S\] implies that $F/S$ is a cover of $L$. By applying the preceding results and this discussion, we obtain the main results.

\begin{thm}\label{dias batten 4.7}
For every free presentation $0\xrightarrow{} R\xrightarrow{} F\xrightarrow{\pi} L\xrightarrow{} 0$ of $L$ and every stem extension $0\xrightarrow{} \ker \U\xrightarrow{} E\xrightarrow{\U} L\xrightarrow{} 0$, one has $Z^*(L) = \overline{\pi}(Z(\overline{F})) = \U(Z(E))$.
\end{thm}

\begin{thm}\label{dias batten 4.9}
Let $Z$ be a central ideal of a finite-dimensional triassociative algebra $L$ and let \[\delta:M(L)\xrightarrow{} (L/L'\otimes Z\oplus Z\otimes L/L')^3\] be as in Theorem \ref{dias batten 4.1}. Then the following are equivalent:
\begin{enumerate}
    \item $\delta$ is the trivial map,
    \item the natural map $\beta$ is injective,
    \item $M(L)\cong \frac{M(L/Z)}{L'\cap Z}$,
    \item $Z\subseteq Z^*(L)$.
\end{enumerate}
\end{thm}

\begin{thm}
Let $L$ be a triassociative algebra and $Z(L)$ be the center of $L$. If $Z(L)\subseteq Z^*(L)$, then $\U(Z(E)) = Z(L)$ for every stem extension $0\xrightarrow{} \ker \U\xrightarrow{} E\xrightarrow{\U} L\xrightarrow{} 0$.
\end{thm}

\end{document}